\documentclass[14pt]{amsart}
\usepackage{amscd,amsthm,amsmath,amssymb}
\usepackage[dvips]{graphicx}
\usepackage[matrix,arrow]{xy}

\makeatletter\@addtoreset{equation}{section}\makeatother

\makeatletter\@addtoreset{subsection}{equation}\makeatother

\newtheorem{theorem}[equation]{Theorem}
\newtheorem{prop}[equation]{Proposition}
\newtheorem{lemma}[equation]{Lemma}

\newtheorem{theorem-definition}[equation]{Theorem-definition}

\theoremstyle{definition}

\theoremstyle{remark}
\newtheorem{remark}[equation]{Remark}

\newcommand{\com}{\mathbb{C}}

\newcommand{\p}{\mathbb{P}}

\newcommand{\ra}{\mathbb{Q}}
\newcommand{\cel}{\mathbb{Z}}

\newcommand{\map}{\longrightarrow}

\thanks{{\it MS 2020 classification}: 14E08, 14M25, 14E30, 14J81}


\begin{document}

\title{Rationality of quotients by finite Heisenberg groups}

\thanks{}

\author{Stanislav Grishin}
\address{\newline{\normalsize Laboratory of AGHA, Moscow Institute of Physics and Technology, 9 Institutskiy per., Dolgoprudny,
Moscow Region, 141701, Russia}
\newline{\it E-mail address}: st.grishin98@yandex.ru}

\author{Ilya Karzhemanov, with an Appendix by Ming\,-\,chang Kang}
\address{\newline{\normalsize Laboratory of AGHA, Moscow Institute of Physics and Technology, 9 Institutskiy per., Dolgoprudny,
Moscow Region, 141701, Russia}
\newline{\it E-mail address}: karzhemanov.iv@mipt.ru}

\address{\newline{\normalsize Department of Mathematics, National Taiwan University, Taipei, Taiwan}
\newline{\it E-mail address}: kang@math.ntu.edu.tw}

\begin{abstract}
We prove rationality of the quotient $\com^n \slash H_n$ for the
finite Heisenberg group $H_n$, any $n \ge 1$, acting on $\com^n$
via its irreducible representation.
\end{abstract}

\sloppy

\maketitle

\bigskip

\section{Introduction}
\label{section:intro}

\refstepcounter{equation}
\subsection{}
\label{subsection:intro-1}

In the present paper, we study rationality of the quotient $\com^n
\slash G$ (\emph{Noether's problem}) for the affine space
$\com^n$, $n \ge 1$, equipped with a linear action of an algebraic
group $G$. Recall that for \emph{finite} $G$ variety $\com^n
\slash G$ can be non\,-\,rational (e.\,g. this is the case for
certain $p$\,-\,groups in \cite{shaf}). At the same time, for
\emph{connected} $G$ the quotient $\com^n \slash G$ is typically
\emph{stably rational}, that is the product $\com^k \times (\com^n
\slash G)$ is rational for some $k$ (see \cite[Theorem
2.1]{fedya}).

Note that variety $\com^n \slash G$ is rational when $G$ is
\emph{Abelian} (see \cite{fis}). Some rationality constructions
for $\com^n \slash G$ with non\,-\,Abelian $G$ can be found in
\cite{prok} (see also \cite{ilya-1}). In the present paper, we
consider a particular case of the \emph{Heisenberg group} $G :=
H_n$ generated by two elements $\xi,\eta$, which act on $\com^n$
as follows (\emph{Schr\"odinger representation}):
$$
\xi: x_i \mapsto \omega^{-i} x_i, \qquad \eta: x_i \mapsto
x_{i+1}\quad (i\in\cel\slash n,\ \omega :=
e^{\frac{2\pi\sqrt{-1}}{n}}),
$$
where $x_1, \ldots, x_n$ form a basis in $\com^n$ (up to a choice
of $\omega$ this is the only irreducible linear representation of
$H_n$).

When studying rationality problem for $\com^n \slash H_n$ it is
reasonable to pass to the \emph{projectivization} and consider the
quotient $X := \p^{n - 1} \slash H_n$ (cf. \cite[Proposition
1.2]{prok}). Here is our main result:

\begin{theorem}
\label{theorem:main} Variety $X$ is \emph{rational} for every $n$.
\end{theorem}

The group $H_n$ is a \emph{central extension} of $\cel \slash n
\oplus \cel\slash n$ by $\cel\slash n \ni [\xi,\eta]$ and so the
action of $H_n$ on $\p^{n - 1}$ factors through that of $\cel
\slash n \oplus \cel\slash n$. Thus Theorem~\ref{theorem:main} is
a natural generalization of \emph{linear} Abelian case mentioned
above. Let us also point out that the case of central extensions
of the \emph{cyclic} groups has been treated in \cite{swa83}.

Our result confirms in addition (a stronger version of) Conjecture
15 in \cite{bt}. Actually, stable rationality of $X$ can be proved
via a direct argument by considering diagonal action of $H_n$ on
$V \times V$, with linear action of $\cel \slash n \oplus
\cel\slash n$ on the second factor. Note also that
Theorem~\ref{theorem:main} is evident when $n \le 3$ and the case
$n = 4$ has been treated in \cite[Theorem 5.2]{prok} (compare with
\cite[Lemma 3.1]{fedya}).

\refstepcounter{equation}
\subsection{}
\label{subsection:intro-2}

Let us outline our approach towards the proof of
Theorem~\ref{theorem:main}. One may observe that the quotient
$\com^n \slash G$ is a \emph{toric} variety for Abelian group $G$.
In our case of $G = H_n$, its action on $\p^{n - 1}$ is also
Abelian, and so it is reasonable to expect that $X$ is toric as
well. This turns out to be (almost) so.

Namely, one employs an instance of the \emph{toric conjecture}
after V.\,V. Shokurov, characterizing toric varieties (with
\emph{Picard number $1$}) in terms of the log pairs: we construct
a $\ra$\,-\,divisor $D$ on $X$ satisfying the assumptions of
Proposition~\ref{theorem:toric} below and reduce rationality
problem for $X$ to that for a \emph{cyclic} quotient of $\p^{n -
1}$ (the latter is rational by the discussion in
{\ref{subsection:intro-1}}). In turn, the explicit action of $H_n$
on $\p^{n - 1}$ allows one to find appropriate invariant divisors
descending to the components of $D$, which is done in
{\ref{subsection:pro-2}}.

Our point was, more generally, to develop a \emph{geometric}
approach to the Noether's problem for central extensions of
Abelian groups (cf. {\ref{subsection:misc-1}} below). Thus the
case of $\com^n \slash H_n$ is a special corollary of this
approach. On the other hand, after our paper appeared online,
Professor Ming\,-\,chang Kang has kindly communicated to us an
\emph{algebraic} proof of Theorem~\ref{theorem:main} (see Appendix
after Section~\ref{section:misc}).

\begin{remark}
\label{remark:quasi-weighted} We show in
Proposition~\ref{theorem:toric} that $X$ is actually a cyclic
quotient of $\p^{n - 1} \slash \widetilde{G}$ for a
\emph{linearized Abelian} group $\widetilde{G}$. Thus $X$
resembles the so\,-\,called \emph{fake weighted projective space}
(see \cite{ksp}). Note however that $X$ \emph{need not} be toric.
Let us consider the first non\,-\,trivial case $n = 3$. Here the
group $H_3$ acts on $\p^2$ preserving the \emph{Hesse pencil}
$\left\{E_t: \, x^3 + y^3 + z^3 + txyz = 0 \ \vert \ t \in \p^1
\right\}$ and on the smooth cubic $E_t$ the $H_3$\,-\,action
coincides with the one of the group of $3$\,-\,torsion points
$E_t[3]$ (see \cite{art-dolg}). The quotient surface $X = \p^2
\slash H_3$ has $4$ singular points of type $A_2$ and so can not
be toric. One may also observe that the algebra of invariants of
$H_3$ in $\com[x,y,z]$ is generated by polynomials $xyz$, $x^3 +
y^3 + z^3$, $x^3y^3 + y^3z^3 + z^3x^3$ and $x^3y^6 + y^3z^6 +
z^3x^6$ (cf. \cite[Section 6]{art-dolg}).
\end{remark}

\bigskip

\section{Proof of Theorem~\ref{theorem:main}}
\label{section:pro}

\refstepcounter{equation}
\subsection{}
\label{subsection:pro-1}

We will be using freely standard notions and facts about the
singularities of pairs (see e.\,g. \cite[Chapter 5]{kol-mor}). All
varieties are assumed to be normal, projective, over $\com$, and
all divisors are $\ra$\,-\,Cartier with rational coefficients.

Our proof of Theorem~\ref{theorem:main} is based on the following:

\begin{prop}[{cf. \cite{kollar-flips}, \cite{prok-tor}, \cite{ilya-tor}}]
\label{theorem:toric} Let $V$ be a $d$\,-\,dimensional variety
with a boundary divisor $D = \displaystyle\sum_{i = 1}^{d + 1} d_i
D_i$, where $D_i$ are prime Weil divisors, such that the following
holds:

\begin{itemize}

    \item the Picard number of $V$ is $1$,

    \smallskip

    \item the log pair $(V,D)$ is log canonical,

    \smallskip

    \item $K_V + D \sim_{\,\ra} 0$,

    \smallskip

    \item $d_iD_i \sim_{\,\ra} d_jD_j$ for all $1 \le i,j \le d + 1$,

    \smallskip

    \item there exists a \emph{finite, \'etale in codimension
    $1$ cyclic} cover $p: V' \map V$ such that
    $p^*(d_iD_i) \sim_{\,\ra} W_i$, $1 \le i \le d + 1$, where $W_i$ are \emph{distinct} Weil divisors on $V'$.

\end{itemize}

Then $V'$ is a \emph{toric} quotient $\p^d \slash \widetilde{G}$
for a finite Abelian group $\widetilde{G}$ with \emph{linearized}
action on $\p^d$. In particular, if $\Gamma \simeq \cel \slash
m\cel$ is the Galois group of $p$, then $V = V' \slash \Gamma$ is
birational to $\p^d \slash \Gamma$ (hence $V$ is rational).
\end{prop}

\begin{proof}
We follow the proof of Lemma 3.1 in \cite{prok-tor}. Namely, after
repeated finite, \'etale in codimension $1$ cyclic covers $V
\stackrel{p}{\longleftarrow V'} \longleftarrow \ldots
\longleftarrow \widetilde{V}$ we obtain a new log pair
$(\widetilde{V}, \widetilde{D} = \varphi^*(D))$, where $\varphi:
\widetilde{V} \map V'$ is the resulting morphism, such that all
$\varphi^*p^*(W_i)$ are \emph{Cartier}. Furthermore, we have
$$
K_{\widetilde{V}} + \widetilde{D} \sim_{\,\ra} \varphi^*p^*(K_V +
D) \sim_{\,\ra} 0
$$
and $(\widetilde{V}, \widetilde{D})$ is log canonical, i.\,e.
$\widetilde{V}$ is a log Fano (note that $\varphi^*p^*(D)$ is
ample).

The Fano index of $\widetilde{V}$ is $\ge d + 1$, since
$-K_{\widetilde{V}} \sim_{\,\ra} \widetilde{D}$ and
$\varphi^*p^*(d_iD_i) \sim_{\,\ra} \varphi^*p^*(d_jD_j)$ for all
$1 \le i,j \le d + 1$. This implies that $\widetilde{V} = \p^d$
and $\varphi$ coincides with the quotient morphism by some finite
group $\widetilde{G}$ (the Galois group of the field extension
$\com(\widetilde{V})\slash\varphi^*\com(V')$). Also, by
construction $\widetilde{G}$ leaves invariant $d + 1$ hyperplanes
$\varphi^*p^*(W_i)$ in $\p^d$, whence it is Abelian.

Further, if $T := (\com^*)^d \subset V'$ is the open torus with
coordinates $z_1, \ldots, z_d$, then for some $N \in \mathbb{N}$
we have: $Np^*(d_iD_i) \sim$ the closure $\widetilde{W}_i$ of
$(z_i = 0) \subset V'$, $1 \le i \le d$, and $Np^*(d_{d + 1}D_{d +
1}) \sim$ the closure $\widetilde{W}_{d + 1}$ of $(z_{d + 1} :=
(z_1 \ldots z_d)^{-1} = 0) \subset V'$. In particular, since each
$p^*(d_iD_i)$ generates the $\ra$\,-\,Picard group of $V'$ and $V'
\setminus T = \displaystyle\bigcup_{i = 1}^{d + 1}
\widetilde{W}_i$ on the toric variety $V'$, up to twist by a
character we may assume that $\Gamma$ either preserves all the
$z_i$ or permutes them cyclicly. In both cases, compactifying $T$
by $\p^d$, we obtain that $T \slash \Gamma$ is birational to the
rational variety $\p^d \slash \Gamma$.
\end{proof}

\refstepcounter{equation}
\subsection{}
\label{subsection:pro-2}

We now turn to the variety $X = \p^{n - 1} \slash H_n$ from
Theorem~\ref{theorem:main}. Let $\pi: \p^{n - 1} \map X$ be the
quotient morphism.

\begin{lemma}
\label{theorem:etale} $\pi$ is \'etale in codimension $1$ and
$K_{\p^{n-1}} \sim_{\,\ra} \pi^*(K_X)$.
\end{lemma}

\begin{proof}
The first assertion follows from the fact that every $\ne 1$
element in $H_n$ has non\,-\,multiple spectrum (see
{\ref{subsection:intro-1}}). Then the equivalence $K_{\p^{n-1}}
\sim_{\,\ra} \pi^*(K_X)$ is the usual Hurwitz formula.
\end{proof}

Identify $x_0,\ldots,x_{n-1}$ from {\ref{subsection:intro-1}} with
projective coordinates on $\p^{n-1}$. Put $f_k :=
\displaystyle\sum_{i \in \cel\slash n} x_i^kx_{i+1}^{n-k}$ for $1
\le k \le n$. We have $\xi^*f_k = \omega^k f_k$ and $\eta^*f_k =
f_k$. Hence polynomials $f_k^n$ are $H_n$\,-\,invariant.

\begin{lemma}
\label{theorem:lin-sys} The linear system $\mathcal{L} \subset
|\mathcal{O}_{\p^{n-1}}(n^2)|$ spanned by $f_1^n, \ldots, f_n^n$
and $(x_0 \ldots x_{n - 1})^n$ is basepoint\,-\,free.
\end{lemma}

\begin{proof}
It suffices to show that $f_1, \ldots, f_n$ and $x_0 \ldots x_{n -
1}$ span a basepoint\,-\,free linear system. Fix an arbitrary $m
\ge n$ and consider the polynomials $f_k^{(m)} :=
\displaystyle\sum_{i \in \cel\slash n} x_i^kx_{i+1}^{m - k}$ for
various $1 \le k \le m$. Let $\mathcal{L}^{(m)}$ be the linear
system spanned by $f_1^{(m)}, \ldots, f_m^{(m)}$ and $x_0 \ldots
x_{n - 1}^{m - n + 1}$. Then we claim that $\mathcal{L}^{(m)}$ is
basepoint\,-\,free (note that $m = n$ corresponds to our case).
Indeed, for $n = 2$ this is trivially true, whereas for $n > 2$ we
restrict to the hyperplanes $(x_i = 0)$ and argue by induction.
\end{proof}

Let $B_1, \ldots, B_n$ be generic elements in the linear system
$\mathcal{L}$ from Lemma~\ref{theorem:lin-sys}. We may assume the
pair $(\p^{n-1}, \displaystyle\sum_{i=1}^n B_i)$ is \emph{log
canonical}.

Further, put $D_i := \pi(B_i)$, $1 \le i \le n$, so that $B_i =
\pi^*(D_i)$,
\begin{equation}
\label{can-cl} K_{\p^{n-1}} + \sum_{i=1}^n B_i \sim_{\,\ra}
\pi^*(K_X + \sum_{i=1}^n D_i)
\end{equation}
(cf. Lemma~\ref{theorem:etale}) and the pair $(X,
\displaystyle\sum_{i=1}^n D_i)$ is also log canonical.

\begin{lemma}
\label{theorem:fact-gal} $\pi$ factorizes as $\p^{n - 1}
\stackrel{q}{\map} X' \stackrel{p}{\map} X$, where $q,p$ are both
\emph{degree $n$, \'etale in codimension $1$ cyclic} covers,
$p^*(d_iD_i) \sim_{\,\ra} W_i$, $1 \le i \le n$, for $d_i :=
1/n^2$ and some \emph{distinct} Weil divisors $W_i$.
\end{lemma}

\begin{proof}
Note that the field extension $\com(\p^{n-1})\slash\pi^*\com(X)$
is Galois with the group $\frak{S} := \cel \slash n \oplus
\cel\slash n$. Restricting to the field of $\xi$\,-\,invariants
yields an intermediate field $\pi^*\com(X) \subset F \subset
\com(\p^{n-1})$. Note that extension $\com(\p^{n-1})\slash F$
corresponds to the quotient morphism $q: \p^{n-1} \map X'$ for $X'
= \p^{n - 1}\slash\left<\xi\right>$ and the cyclic subgroup
$\left<\xi\right> \subset \frak{S}$. Finally, $\com(X)\slash F$ is
also Galois, corresponding to the quotient morphism $p: X' \map X
= X'\slash\left<\eta\right>$.

Further, consider the divisors $B_0 := ((x_0 \ldots x_{n - 1})^n =
0)$ and $H_i := (x_i = 0)$, $0 \le i \le n - 1$, so that $B_0 =
\displaystyle n\sum_{i = 0}^{n - 1} H_i$. We have $\displaystyle
\frac{1}{n^2}\,B_0 \sim_{\,\ra} H_i$ for all $i$ and hence
$$
q_*(\frac{1}{n^2}\,B_0) \sim_{\,\ra} q_*(H_i) \sim_{\,\ra} nq(H_i)
$$
because $q_*(H_i) = nq(H_i)$ for $H_i$ being $\xi$\,-\,invariant
hyperplanes. This implies that
$$
p^*(\frac{1}{n^2}\,D_i) = \frac{1}{n^2}\,q(B_i) =
\frac{1}{n^3}\,q_*(B_i) \sim_{\,\ra} q_*(\frac{1}{n^3}\,B_0)
\sim_{\,\ra} q(H_{i - 1}) =: W_i
$$
for all $1 \le i \le n$.
\end{proof}

Put $D := \displaystyle\frac{1}{n^2}\sum_{i = 1}^n D_i$. Then it
follows immediately from \eqref{can-cl} and
Lemma~\ref{theorem:fact-gal} that the log pair $(X,D)$ satisfies
all the assumptions in Proposition~\ref{theorem:toric} (for $V :=
X$). Thus $X$ is rational and the proof of
Theorem~\ref{theorem:main} is complete.

\bigskip

\section{Miscellany}
\label{section:misc}

\refstepcounter{equation}
\subsection{}
\label{subsection:misc-1}

It would be interesting to extend the technique presented in
Section~\ref{section:pro} to the case of quotients $\p^{n - 1}
\slash G$ by other finite central extensions of Abelian groups.
This requires, however, analogs of technical lemmas from
{\ref{subsection:pro-2}}, where we have crucially used that $G =
H_n$.

More generally, it would be interesting to give a characterization
of those finite groups $G$, for which $\p^{n - 1} \slash G$ is a
cyclic quotient of toric variety (cf.
Remark~\ref{remark:quasi-weighted}). Observe at this point that
the singularities of $X = \p^{n - 1} \slash H_n$ are
\emph{non\,-\,exceptional} (cf. \cite[Proposition 3.4]{mar-pro}),
and one might try to look for a similar property, distinguishing
(cyclic quotients of) toric $\p^{n - 1} \slash G$.

\refstepcounter{equation}
\subsection{}
\label{subsection:misc-2}

Initially, our interest was in constructing a \emph{mirror dual}
$Y^+$ for Calabi\,--\,Yau threefolds $Y$, studied in
\cite{gro-pop}. Recall that $Y$ is a small resolution of a nodal
Calabi\,--\,Yau $V \subset \p^{n - 1}$, \emph{invariant under
$H_n$}, such that there is a pencil of \emph{$(1,n)$\,-\,polarized
Abelian surfaces} $A \subset V$. The action of $H_n$ extends to a
\emph{free} one on $Y$ and it is expected that $Y^+ = Y \slash
H_n$. Indeed, when $n = 8$ the \emph{derived equivalence} between
$Y$ and $Y \slash H_n$ was established in \cite{schne}, which on
the level of Abelian surfaces is the Mukai equivalence between $A$
and $\text{Pic}^0(A) = A \slash H_n$ (note that $H_n$ acts on $A$
via shifts by $n$\,-\,torsion points).

In particular, when $n = 5$ and $V$ is the
\emph{Horrocks\,--\,Mumford quintic} (see \cite[Section
3]{gro-pop}), $V \slash H_5$ is a Calabi\,--\,Yau hypersurface in
(almost) toric variety $\p^4 \slash H_5$. This brings in a
possibility for applying Batyrev's construction of mirror pairs
(see \cite{bat}) as well as other explicit methods: matrix
factorizations, period integrals, etc. (see e.\,g. \cite{polisch},
\cite{yau}). We plan to return to this subject elsewhere.

\refstepcounter{equation}
\subsection{}
\label{subsection:misc-3}

As a complement to {\ref{subsection:misc-1}}, one may try to
attack (stable) rationality problem for various quotients $\p^{n -
1} \slash G$ by considering their classes $[\p^{n - 1} \slash G]$
in $K_0(\text{Var})$, the \emph{Grothendieck ring} of complex
algebraic varieties, and applying \cite[Corollary 2.6]{la-lu} to
them. It is important however that $[\p^{n - 1} \slash G]$ be
\emph{non\,-\,zero} modulo $\mathbb{L} := [\mathbb{A}^1]$ (see
\cite{karz-cut} for some examples of varieties $Z$ with $[Z] = 0
\mod \mathbb{L}$).

It is trivially true that $[\p^{n - 1} \slash H_n]$ is not
divisible by $\mathbb{L}$ (for it equals $[\p^{n - 1}]$ modulo
$\mathbb{L}$ by Theorem~\ref{theorem:main}) and it would be
interesting to find out whether this is always the case for
\emph{all} quotients $\p^{n - 1} \slash G$. Perhaps the fact that
any such variety is \emph{stably b\,-\,inf. trans.} (see
\cite[Corollary 3.2]{bkk}) might be of some use here.

\bigskip

\section*{Appendix by Ming\,-\,chang Kang}

The algebraic proof of Theorem~\ref{theorem:main} follows the same
lines as in \cite{chu-kang}. Namely, put $\lambda := [\xi, \eta]$
and $y_0 := x_0^n$, $y_i := x_i \slash x_{i - 1}$, $1 \le i \le n
- 1$ (see {\ref{subsection:intro-1}}). Then we have $\com(x_0,
\ldots, x_{n - 1})^{\left<\lambda\right>} = \com(y_0, \ldots, y_{n
- 1})$. By \cite[Theorem 4.1]{chu-kang} it suffices to prove
rationality of $\com(y_1, \ldots, y_{n -
1})^{\left<\xi,\eta\right>}$. Note that the action of $\xi$ on
$y_i$, $1 \le i \le n - 1$, is given by $\xi: y_i \mapsto \omega
y_i$.

Define $z_1 := y_1^n$, $z_i := y_i \slash y_{i - 1}$, $2 \le i \le
n - 1$. Then we have $\com(y_1, \ldots, y_{n -
1})^{\left<\xi\right>} = \com(z_1, \ldots, z_{n - 1})$. Note that
the action of $\eta$ on $z_i$ is the same as the action of $\tau$
in \cite[p. 686]{chu-kang} (by replacing $p$ with $n$ everywhere).

Now define $w_1 := z_2$, $w_i := \eta^{i - 1}(z_2)$, $2 \le i \le
n - 1$. Then we have $\com(z_1, \ldots, z_{n - 1}) = \com(w_1,
\ldots, w_{n - 1})$ and the action of $\eta$ is as follows:
$$
\eta: w_1 \mapsto w_2 \mapsto w_3 \mapsto \ldots \mapsto w_{n - 1}
\mapsto \frac{1}{w_1,\ldots w_{n - 1}}.
$$
The latter action can be linearized exactly as in the middle of
\cite[p. 687]{chu-kang}. Hence we can apply Fischer's Theorem (see
\cite{fis}).

\bigskip

\thanks{{\bf Acknowledgments.}
It is our pleasure to thank A. Belavin, F. Bogomolov, and Ilya
Zhdanovskiy for valuable comments. The work was partially
supported by the Russian Academic Excellence Project 5\,-\,100.

\end{document}